\documentclass[a4,12pt]{amsart}
\usepackage{amssymb,amsmath,amsthm,amsfonts}
\usepackage{mathrsfs,eufrak,euscript}
\usepackage[active]{srcltx}
\newtheorem{thm}{Theorem}[section]
\newtheorem{defn}[thm]{Definition}
\newtheorem{lemma}[thm]{Lemma}
\newtheorem{ex}[thm]{Exercise}
\newtheorem{cor}[thm]{Corollary}
\newtheorem{prop}[thm]{Proposition}
\theoremstyle{definition}

\newtheorem{rmk}[thm]{Remark}
\newtheorem{eg}[thm]{Example}
\newtheorem{qn}[thm]{Question}
\numberwithin{equation}{section}

\def\<{\langle}
\def\>{\rangle}

\newtheorem*{oldproof}{Proof}
\renewenvironment{proof}[1][{}]{\begin{oldproof}[#1]}{\qed\end{oldproof}}

\title{\underline{Weak Topologies}}
\author{G. Ramesh}

\address{Department of Mathematics\\I. I. T. Hyderabad\\Kandi, Sangareddy, Telangana, India-502 285.}
\email{rameshg@math.iith.ac.in}
\date{\today}
\thanks{ This notes were written for Annual Foundation School (AFS) - II (2016) held at Kerala School of Mathematics, India. \\ I am thankful to Dr. Sudeshna Basu for her suggestions regarding  Theorem 5.1}

\begin{document}
\maketitle
\tableofcontents
\section{Topological preliminaries}

We discuss about the weak and weak star topologies on a normed linear space. Our aim is to prove the well known Banach-Alaouglu theorem and discuss some of its consequences, in particular,  characterizations of reflexive spaces.

In this section we briefly revise topological concepts that are required for our purpose.

\begin{defn}
Let $X$ be a non empty set and  $\tau \subseteq \mathcal P(X)$, the power set of $X$. Then  $\tau$ is said to be a topology on $X$ if the following conditions holds true:
\begin{enumerate}
\item $\emptyset,\; X\in \tau$
\item if $U_{\alpha}\in \tau,\; \alpha\in \Lambda$, then $\displaystyle \bigcup_{\alpha \in \Lambda}U_{\alpha}\in \tau$
\item if $F_i\in \tau$ for $i=1,2,\dots, n$, then $\displaystyle \bigcap_{i=1}^nF_i\in \tau$.
\end{enumerate}
The pair $(X,\tau)$ is said to be a topological space.
\end{defn}
The members of $\tau$ are called open sets.

Let $X$ be a non empty set. Then $\tau_{i}={\{{\{\emptyset}\}, X}\}$ and $\tau_d=\mathcal P(X)$  are topologies and are called as the indiscrete topology and the discrete topology, respectively. Note that if $\tau$ is any other topology on $X$, then $\tau_i\subset \tau \subset \tau_d$.

Let $\tau_1$ and $\tau_2$ be two topologies on $X$. Then we say that $\tau_1$ is weaker than $\tau_2$ if $\tau_1\subseteq \tau_2$. In this case, $\tau_2$ is said to be stronger than $\tau_1$. Clearly, the weak topology contains less number of open sets than the stronger topology.

We recall the following two occasions:
\begin{enumerate}
\item[(i)] Let $H=\ell^2$ and ${\{e_n:n\in \mathbb N}\}$ be the standard orthonormal basis for $H$. Recall that $e_n(m)=\delta_n(m)$, the Dirac delta function. Then for $m,n\in \mathbb N$, we have $\|e_n-e_m\|_2=\sqrt{2}$. Hence the sequence $(e_n)$ is not convergent.
\item The unit ball $B_{H}={\{x\in H:\|x\|_2=1}\}$ is not compact. In general, the unit ball $B_X$ in a normed linear space $X$ is compact if and only if the normed linear space $X$ is finite dimensional.
\end{enumerate}

The above examples suggests that in  infinite dimensional spaces, the compact sets and convergent sequences are a few. Since compact sets are very important in Mathematical Analysis and Applications, it is not suggestible to consider the norm topology. Hence we look at weaker topologies so that we can get more compact spaces and convergent sequences.

Before discussing the weak topologies on normed linear spaces we recall the notions of basis and subbasis and some relevant concepts which we need later.

A collection $\mathcal B\subseteq \mathcal P(X)$ is called a basis for a topology on $X$ if and only if
\begin{enumerate}
\item for every $x\in X$, there exists $B\in \mathcal B$ such that $x\in B$
\item if $B_1,B_2\in \mathcal B$ and $x\in B_1\cap B_2$, then there exists $B\in \mathcal B$ such that $x\in B\subseteq B_1\cap B_2$.
\end{enumerate}

The topology generated by $\mathcal B$ is given by the following rule:\\
A set $U$ in $X$ is said to be open in $X$ if for all $x\in U$, there exists $B\in \mathcal B$ such that $x\in B\subseteq U$.

In this case the topology generated by $\mathcal B$ is given by $\tau_{\mathcal B}=\displaystyle \bigcup_{B\in \mathcal B} B  $.

Let $(X,\tau)$ be a topological space and $\mathcal B\subseteq \tau$. Then $\mathcal B$ is a basis for $\tau$ if for every $U\in\tau$ and $x\in U$, there
exists $B\in \mathcal B$ such that $x\in B\subseteq U$.

Let $\mathcal S\subset \mathcal P(X)$. Then the smallest topology containing $\mathcal S$ is the intersection of all topologies on $X$ containing $\mathcal S$ (the discrete topology is one such topology). In fact, this is nothing but the topology generated by $\mathcal S$.

Let $\mathcal B$ and $\mathcal B'$ be bases for the topologies $\tau$ and $\tau'$, respectively. Then the following are equivalent:
\begin{enumerate}
  \item $\tau \subseteq \tau'$
  \item if $B\in \mathcal B$ and $x\in B$, there exists $B' \in \mathcal B'$ such that $x\in B'\subseteq B$.
\end{enumerate}

\begin{defn} Let $X$ be a set. Then a sub-basis is a collection $\mathcal S\subseteq \mathcal P(X)$ such that $\displaystyle \bigcup_{S\in \mathcal S} S=X$. The
topology generated by $\mathcal S$  is defined by the rule: $ U\subseteq  X$ is open if for each $x\in  U$ we can  find  $S_1, S_2,\dots, S_n\in \mathcal S$  such that $x \in S_1 \cap  S_2\cap\dots \cap S_n  \subset U$.
\end{defn}

\begin{defn}
Let $X$ be a topological space with topology $\tau$. Then $\mathcal S\subseteq \tau$ is a sub-basis for
$X$ if for each open set $U$ and each $x\in U$, there exist finitely many $S_1,S_2,\dots, S_n \in \mathcal  S$ such that
$x \in S_1\cap S_2\cap \dots \cap S_n\subset U$.

\end{defn}

If $\mathcal S$ is a subbasis for a topology $\tau$ on $X$, then $$\mathcal B:=\Big\{ \bigcap_{\alpha \in F} S_{\alpha}: F \; \text{is a finite subset of}\;  \Lambda  \Big\}$$
forms a basis for $\tau$. This means that $$\tau=\Big\{\bigcup_{F\in \mathcal F} \; \bigcap_{S\in F}S: \mathcal F \subset  \mathcal S^{<w}\Big\}\bigcup {\{\emptyset, X}\},$$
where $ \mathcal S^{<w}$ is the set of all finite subsets of $\mathcal S$.

The topology generated by the subbasis $\mathcal S$ is the smallest topology containing $\mathcal S$.

Here we list out some more definitions and basic results:
\begin{enumerate}
\item In a topological space $(X,\tau)$ a set $N\subset X$ is said to be a neighbourhood of a point $x\in X$ if there exists an open set $U$ such that $x\in U\subset N$.\\
\item A subset $A$ of $X$ is closed, if  $A^{c}$, the  complement  of $A$ is open. Since the intersection of a collection of closed sets is closed, every subset $A$ of $X$ is contained in a unique minimal closed  set
$\overline{A}={\{x\in X: \text{every neighbourhood of} \; x\; \text{intersects}\; A}\}$, called the closure of $A$.

\item A topological space $(X,\tau)$ is said to be  \textit{Hausdorff} if  for any two points $x,y\in X$ with $x\neq y$, there exists two disjoint open sets $U_x$ and $U_y$ such that $x\in U_x$ and $y\in U_y$.\\

\item A sequence $(x_n)$ in $X$ is said to converge to $x\in X$, if for every open set $U$ of $x$, there exists $n_0\in \mathbb N$ such that $x_n\in U$ for all $n\geq n_0$. Note that if $X$ is Hausdorff, then $x_n$ converge to a unique limit in $X$.\\



   \item Let $(X,\tau_1)$ and $(Y,\tau_2)$ be topological spaces and $f:X\rightarrow Y$ be a function. Then $f$ is said to be continuous if for every open set $V\in \tau_1$, $f^{-1}(V)\in \tau_1$.\\

   \item Let $(X_1,\tau_1)$ and $(X_2,\tau_2)$ be two topological spaces.  Then the collection
   $\mathcal B:={\{U\times V: U\in \tau_1,V\in \tau_2 }\}$ forms a basis for a topology on $X_1\times X_2$, and is called as the product topology. \\

   \item If $(X_i,\tau_i)$ is  topological space for $i=1,2,\dots,n$. Then we can define the product topology by the similar way.\\

   \item Let ${\{X_{\alpha}}\}_{\alpha \in J}$ be a family of topological spaces. Let $X=\displaystyle \prod_{\alpha\in J}X_{\alpha}$ and $\pi_{\beta}:X\rightarrow X_{\beta}$ be the projection defined by $$\pi_{\beta}((x_{\alpha})_{\alpha \in J})=x_{\beta}.$$ Let $S_{\beta}={\{\pi_{\beta}^{-1}(U_{\beta})| U_{\beta} \; \text{is open in }\; X_{\beta}}\}$ and $\mathcal S=\bigcup_{\beta \in J}\mathcal S_{\beta}$. The topology generated by the subbasis $\mathcal S$ is called the product topology and the space $X$ with this topology is called the product space.\\
       \item Let $(X,\tau)$ be a topological space. A collection $\mathcal A = {\{U_i}\}_{i\in I}$ of open subsets of $X$ such that
$X = \cup_{i\in I}U_i$ is called an open cover of $X$. A subcollection $\mathcal B\subseteq \mathcal A$ is called a subcover if the
union of the sets in $\mathcal B$  also cover $X$. A subcover is finite if it contains finitely many open sets.\\

\item Let $(X,\tau)$ be a topological space. Then $X$ is compact if every open cover has a finite subcover.\\

  \item  (Tychonoff theorem) If each $X_{\alpha}$ is compact, then $X=\displaystyle \prod_{\alpha \in J}X_{\alpha}$ is compact with respect to the product topology.

\end{enumerate}
\section{Weak Topology}
Let $X$ be a non empty set and $(X_{\alpha},\tau_{\alpha})_{\alpha\in \Lambda}$ be a family of topological spaces.  Let $\mathcal F:={\{f_{\alpha}: X\rightarrow X_{\alpha}: \alpha \in \Lambda}\}$.

\begin{qn}
Define a topology $\tau$ on $X$ in such a way that each $f\in \mathcal F$ is continuous with respect to $\tau$.
\end{qn}

Suppose if we take $\tau$ to be the discrete topology, then each $f\in \mathcal F$ is continuous with respect to $\tau$. But, we know that the eventually constant sequences are the only convergent sequences and only finite sets are compact in this topology, so we avoid this. Hence we rephrase the above question as follows.

\begin{qn}
Find the smallest topology $\tau$ such that each $f\in \mathcal F$ is continuous with respect to $\tau$.
\end{qn}

Since we expect each $f_{\alpha}$ to be continuous with respect to $\tau$, if $U_{\alpha}\in \tau_{\alpha}$, then $f_{\alpha}^{-1}(U_{\alpha})\in \tau$. If $\tau$ is topology, then arbitrary unions of finite intersection of $f_{\alpha}^{-1}(U_{\alpha})$ should be in $\tau$.
Let $\mathcal S:={\{f_{\alpha}^{-1}(U_{\alpha}): U_{\alpha }\in \tau_{\alpha},\; \alpha \in \Lambda}\}$. Then the topology generated by $\mathcal S$ is the required topology. This topology is called the weak topology generated by $\mathcal F$.
\begin{defn}
Let $X$ be a non empty set and $(X_\alpha,\tau_\alpha)$ be a family of topological spaces indexed by $\Lambda$. The weak topology generated by the family of functions $\mathcal F={\{f_\alpha: X\rightarrow X_\alpha}\}$ is the topology generated by the subbasic open sets ${\{f^{-1}_\alpha(U_\alpha):U_\alpha \in \tau_\alpha,\; \alpha \in \Lambda}\}$.
\end{defn}
We denote the topology generated by $\mathcal F$ on $X$ by $\sigma(X,\mathcal F)$.

\begin{rmk}
 Let $\mathcal F_1$ and $\mathcal F_2$ be family of functions from $X$ into $Y_{\alpha}\; (\alpha \in \Lambda)$ such that $\mathcal F_1 \subseteq \mathcal F_2$. Then $\sigma(X,\mathcal F_1)$ is weaker than $\sigma(X,\mathcal F_2)$.

\end{rmk}

\subsection{Examples}
\begin{enumerate}
\item  Let $(X,d)$ be a metric space and $u\in X$. Define $f_u:X\rightarrow \mathbb R$ by
\begin{equation*}
f_u(x)=d(x,u),\; x\in X.
\end{equation*}

Let $\mathcal F={\{f_u:u\in X}\}$. We would like to find out $\sigma(X,\mathcal F)$.  Let $r>0$. Then
 \begin{align*}
 f_u^{-1}(-r,r))&={\{x\in X: f_u(x)<r}\}\\
                       &={\{x\in X: d(u,x)<r}\}\\
                       &=B(u,r),
 \end{align*}
  the open ball centered at $u$ with radius $r$. Hence the topology generated by the family of functions $\mathcal F$ is the metric topology.\\

\item  Let $X=C[a,b]:={\{f:[a,b]\rightarrow \mathbb R: f\; \text{ is continuous}}\}\; (a,b\in \mathbb R, \; a<b)$ and $x\in X$. Define $e_x:X\rightarrow \mathbb R$ by
\begin{equation*}
e_x(f)=f(x),\; f\in X.
\end{equation*}
Let $\mathcal F={\{e_x:x\in X}\}$. Then  $\sigma(X,\mathcal F)$, the weak topology generated by $\mathcal F$ is called the topology of pointwise convergence on $X$.\\

\item  Let $(X_{\alpha},\tau_{\alpha}) \;(\alpha \in \Lambda)$ be topological spaces and $X=\displaystyle \Pi_{\alpha \in \Lambda}X_\alpha$. That is,
\begin{equation*}
X={\{x:\Lambda \rightarrow \displaystyle \cup X_{\alpha}: x(\alpha)\in X_\alpha}\}.
\end{equation*}
Note that $X\neq \emptyset$, by the Axiom of Choice. Let $x\in X$. Then $x=(x_\alpha)_{\alpha \in \Lambda}$. Define $P_\alpha:X\rightarrow X_\alpha$ by
$P_\alpha(x)=x_\alpha$. Then each $P_\alpha$ is a projection. The weak topology generated by $\mathcal F={\{P_\alpha:\alpha \in \Lambda}\}$ has the subbasis of the form
\begin{equation*}
{\{P_\alpha^{-1}(U_\alpha):\alpha \in \Lambda,\, U_\alpha \in \tau_\alpha}\}.
\end{equation*}
But this is the product topology on $X$. Hence $\sigma(X,\mathcal F)$ coincide with the product topology on $X$.
\end{enumerate}

\subsection{Properties}
\begin{thm}
Let $(X,\tau)$ be a topological space and $\mathcal F={\{f_{\alpha}: X\rightarrow X_{\alpha}}\}$ be a family of continuous functions from $X$ into the topological space $(X,\tau_{\alpha})$. Then $\sigma(X,\mathcal F)$, the weak topology generated by $\mathcal F$ is weaker than $\tau$.
\end{thm}

Let $X_{\alpha}=\mathbb F \; (\mathbb F=\mathbb R\; \text{or}\; \mathbb C)$ and $f_{\alpha}: X\rightarrow X_{\alpha}$ be a function. Let $\mathcal F:={\{f_{\alpha}: \alpha \in \Lambda}\}$. Let $\mathcal G:={\{F\subset \Lambda: F\; \text{ is finite}}\}$. Then the weak topology generated by $\mathcal F$ has the basis $$\mathcal B:=\Big\{\displaystyle \bigcap_{\alpha \in F}f_{\alpha}^{-1}((-\epsilon, \epsilon)): F\in \mathcal G,\; \epsilon>0 \Big\}.$$

Hence a set $U$ is open in the weak topology of $X$ if and only if given $x\in U$, there exists $\alpha_1,\alpha_2,\dots, \alpha_n\in \Lambda$ such that $x\in \displaystyle \cap_{i=1}^nf_{\alpha_i}^{-1}(-\epsilon_i,\epsilon_i)\subset U$. This is implies that  $|f_{\alpha_i}(x)|<\epsilon_i $ for $i=1,2,\dots, n$.

A subbasic open set containing a point $x_0\in X$ is of the form $f_{\alpha}^{-1}\Big( f_{\alpha}(x_0)-\epsilon,f_{\alpha}(x_0)-\epsilon,\Big)$ for each $\alpha \in \Lambda$ and for each $\epsilon >0$. In fact, this is nothing but ${\{x\in X: |f_{\alpha}(x)-f_{\alpha}(x_0)|<\epsilon}\}$, for all choices of $\alpha \in \Lambda$ and for all choices of $\epsilon>0$.

A basic open set containing $x_0\in X$ is nothing but the finite intersection of subbasic open sets. Hence it can be of the form:\\
\begin{equation*}
B(x_0; f_1,f_2,\dots, f_n: \epsilon):={\{x\in X: |f_i(x)-f_i(x_0)|<\epsilon}\},
\end{equation*}
for $f_i\in \mathcal F, \;  i=1,\dots, n,\; (n\geq 1)$  and  $\epsilon>0$.

The following properties can be proved easily.

\begin{prop}\label{weakopensetsprop}
\begin{enumerate}
\item $x_0\in B(x_0; f_1,f_2,\dots,f_n; \epsilon)$
\item Let  $g_i\in \mathcal F, \; i=1,2,\dots,m$. Let $B_1=B(x_0; f_1,f_2,\dots,f_n; \epsilon_1),\; B_2=B(x_0; g_1,g_2,\dots,g_m; \epsilon_2)$ and $\epsilon=\min{\{\epsilon_1,\epsilon_2}\}$. Then
\begin{equation*}
B(x_0; f_1,f_2,\dots,f_n,g_1,g_2,\dots,g_m;\epsilon)\subseteq B_1 \cap B_2.\\
\end{equation*}
\item if $x_1\in B(x_0, f_1,f_2,\dots, f_n; \epsilon)$, then there exists $\delta>0$ such that $B(x_1, f_1,f_2,\dots, f_n; \delta)\subset B(x_0, f_1,f_2,\dots, f_n; \epsilon)$\\
\item \label{Hausdorffproperty} Let $x_0,y_0\in X$ with $x_0\neq y_0$ and $\epsilon_1,\epsilon_2>0$. Let $x\in B(x_0: f_1,f_2,\dots,f_n; \epsilon_1)\cap B(y_0: g_1,g_2,\dots,g_m; \epsilon_2)$.  Then there exists $\delta>0$ such that
$B(x: f_1,f_2,\dots,f_n,g_1,g_2,\dots,g_m; \delta)\subseteq B_1\cap B _2$\\
\item Let $x_0,y_0\in X$ with $x_0\neq y_0$. Assume that there exists $f\in \mathcal F$ such that $f(x_0)\neq f(y_0)$. Then $B\big(x_0,f, \dfrac{\epsilon}{2}\big)\cap B\big(y_0, f,\dfrac{\epsilon}{2}\big)=\emptyset$.
\end{enumerate}
\end{prop}

From now onwards, we can consider weak topologies on normed linear spaces.

Let $X$ be  a normed linear space over a field $\mathcal F$ (here $\mathbb F=\mathbb C$ or $\mathbb R $) and $\mathcal F=X^*$, the dual space of $X$. Recall that $$X^*:={\{f:X\rightarrow \mathbb C: f\; \text{ is linear and continuous}}\}.$$
\begin{defn}
Let $X$ be a normed linear space and $X^*$ denote the dual space of $X$. Then the weak topology  on $X$ is  $\sigma(X,X^*)$, the weak topology generated by members of $X^*$.
\end{defn}

Note that a subbasic open set about $x_0\in X$ is of the form:
\begin{equation*}
V(x_0,f,\epsilon):={\{x\in X: |f(x)-f(x_0)|<\epsilon}\}\; \text{for all}\; f\in X^*,\; \text{and}\; \epsilon>0
\end{equation*}
and a basic open set about $x_0\in X$ in the topology $\sigma(X,X^*)$ is of the form:
\begin{equation*}
B(x_0;f_1,f_2,\dots,f_n; \epsilon)=\big\{x\in X: |f_i(x)-f_i(x_0)|<\epsilon\big\},
\end{equation*}
for all choices of $f_i\in X^*,\; i=1,2,\dots,n\; (n\geq 1)$ and  $\epsilon>0$.

Note that the topology $\sigma(X,X^*)$ is weaker than the norm topology of $X$. That is, every weakly open set is open with respect to the norm topology.

\begin{prop}
  The topology $\sigma(X,X^*)$ is Hausdorff.
\end{prop}
\begin{proof}
  Since $X\neq {\{0}\}$, choose $x\in X$. Then by the Hahn-Banach theorem there exists a $0\neq f\in X^*$ such that $f(x)=\|x\|$. Now, the conclusion follows by Property (\ref{Hausdorffproperty}) of Proposition \ref{weakopensetsprop}.
\end{proof}

Let $X$ be a normed linear space and a sequence $(x_n)\subset X$  converges to $x\in X$ weakly, then we denote this by $x_n\xrightarrow{w}x$.

\begin{thm}\label{weakconvergence}
 Let $X$ be a normed linear space and  $(x_n)\subset X$. Then $x_n\xrightarrow{w}x\in X$ if and only $f(x_n)$ converges to $f(x)$ for all $f\in X^*$.
\end{thm}

\begin{proof}

   Suppose $x_n\xrightarrow{w}x$. Let $f\in X^*$ and $\epsilon>0$ be given. Then $U={\{y\in X: |f(y)-f(x)|<\epsilon}\}$ is a weakly open set containing $x$. Choose $n_0\in \mathbb N$
such that  $x_n \in U $ for all $n\geq n_0$. From this we can  conclude that $f(x_n)\rightarrow f(x)$.

To prove the other way, assume that $f(x_n)\rightarrow f(x)$ for all $f\in X^*$. Let $U$ be a weakly open set containing $x$. Choose a $\delta>0$  and $f_i\in X^*,\; i=1,2,\dots,m$ such that
 \begin{equation*}
 \big\{y\in X: |f_i(x)-f_i(y)|<\epsilon:i=1,2,\dots,m \big\} \subset U.
 \end{equation*}
 As $f_j(x_n)\rightarrow f_j(x)$ for each $j=1,2,\dots,m$, there exists $n_j\in \mathbb N$ such that $|f_j(x_n)-f_j(x)|<\epsilon$ for all $n\geq n_j$. Let $n_0=\max{\{n_j:j=1,2,\dots,m}\}$. Then for $n\geq n_0$, we have

\begin{equation*}
|f_j(x_n)-f_j(x)|<\epsilon \; \text{ for all} \; n\geq n_0.
\end{equation*}
That is  $x_n\in U$ for all $n\geq n_0$. Hence $x_n\xrightarrow {w}x$.
\end{proof}

\begin{rmk}
  By the same arguments as above, we can prove  Theorem \ref{weakconvergence} by replacing nets in place of sequences.
\end{rmk}
\begin{ex}
Show that if $x_n\xrightarrow{w}x$ and $y_n\xrightarrow{w}y$ and $\alpha \in \mathbb C$, then
\begin{itemize}
\item[(i)] $x_n+y_n\xrightarrow{w}x+y$
\item[(ii)]$\alpha x_n\xrightarrow{w}\alpha x$.
\end{itemize}
\end{ex}
\begin{ex}
Let $H$ be a Hilbert space and $(x_n)\subset H$. Show that $x_n\xrightarrow{w}x$ if and only if $\langle x_n,y\rangle \rightarrow \langle x,y\rangle$ for every $y\in H$.
\end{ex}
\begin{rmk}
If $H$ is a separable Hilbert space with a countable orthonormal basis ${\{e_n:n\in \mathbb N}\}$, then $e_n\xrightarrow{w}0$ as $n\rightarrow \infty$.
\end{rmk}

Here we recall some of the basic theorems from Functional Analysis, which we need later.

\begin{prop}
Let $X$ be a non zero normed linear space. If $x\in X$, then there exists $0\neq f\in X^*$ such that $f(x)=\|x\|$ and $\|f\|=1$. In other words,  we can say that if  $X\neq {\{0}\}$, then $X^*\neq {\{0}\}$.
\end{prop}
\begin{thm}(Hahn-Banach Separation theorem)
Let $A$ and $B$ be disjoint non empty convex sets in a real Banach space $X$. Then $A$ and $B$ can be separated by a hyperplane, i.e. there exists a non zero functional $f\in X^*$ and a real number $c$ such that
\begin{equation*}
f(x)\leq c< f(y)\;\; \text{ for all}\;  x\in A \; \text{ and}\;  y\in B.
\end{equation*}
\end{thm}
\begin{thm}(Banach-Steinhauss theorem)\label{bsthm}
Let $X$ be a Banach space and $Y$ be a normed linear space over a field $\mathbb K$. Suppose $(T_n)\subset \mathcal B(X,Y)$, the space of all bounded linear maps from $X$ to $Y$. Assume that  $T_nx\rightarrow Tx$ for all $x\in X$. Then $T\in \mathcal B(X,Y)$ and $\|T\|\leq \displaystyle \liminf_{n}\|T_n\|$.
\end{thm}
\begin{prop}
Let $X$ be a Banach space. Suppose $(x_n)\subset X$ be such that $x_n\xrightarrow{w} x$. Then
\begin{itemize}
\item[(a)] $(x_n)$ is bounded and
\item[(b)] $\|x\|\leq \displaystyle \liminf_{n}\|x_n\|$.
\end{itemize}
\end{prop}
\begin{proof}
Let $e_{x_n}: X^*\rightarrow \mathbb K$ be the evaluation map. That is $e_{x_n}(f)=f(x_n)$ for all $f\in X^*$. Since $x_n\xrightarrow{w} x$, it follows that $e_{x_n}(f)\rightarrow e_{x}(f)$ for all $f\in X^*$. Hence by the Banach-Steinhauss theorem, we can conclude that $e_{x}$ is bounded and $\|e_x\|\leq \displaystyle \liminf_{n}\|e_{x_n}\|$. Now, the result follows as $\|e_y\|=\|y\|$ for any $y\in X$.
\end{proof}

\begin{ex}
Let $(x_n)$ be a sequence in an inner product $X$ such that $x_n\xrightarrow{w}x\in X $ and $\|x_n\|\rightarrow \|x\|$ as $n\rightarrow \infty$. Show that $x_n\rightarrow x$.
\end{ex}
\begin{defn}
Let $(X,d)$ be a metric space and $f:X\rightarrow \mathbb R$ be a function. Then $f$ is said to be \textit{lower semi continuous} (\textit{weakly lower semicontinuous}) if $x_n\rightarrow x$ ($x_n\xrightarrow{w}x$) implies $f(x)\leq \liminf_{n} f(x_n)$.
\end{defn}
We can conclude that the function $\|\cdot\|$ is a weakly lower semicontinuous  on a normed linear space with respect to the weak topology.

\begin{defn}
A topological vector space is a vector space $V$ over a field $\mathbb K$ with a topology $\tau$ on it such that the vector space operations $+$ (vector addition) and $\cdot$  (scalar multiplication) are continuous with respect to the topology $\tau$.
\end{defn}
\begin{eg}
Let $X$ be a normed linear space over a field $\mathbb K$. Then $(X,\sigma(X,X^*))$ is a topological vector space.
\end{eg}

Now we discuss about bounded sets in the weak topology.

\begin{defn}(weakly bounded set)
Let $X$ be a normed linear space and $A\subset X$. Then $A$ is said to be weakly bounded if $f(A)$ is bounded for all $f\in X^*$. That is,  for each $f\in X^*$, there exists  $M_f>0$ such that $|f(a)|\leq M_{f}$ for all $a\in A$.
\end{defn}

\begin{rmk}
Every bounded set in a normed linear space is weakly bounded. To see this, let $A$ be a bounded subset of $X$. Then there exists a positive number $K$ such that $\|a\|\leq K$ for all $a\in A$. For $f\in X^*$, we have that
\begin{equation*}
|f(a)|\leq \|f\|\|a\|\leq K\|f\|.
\end{equation*}
Taking  $M_f=K\|f\|$, we can conclude that $A$ is weakly bounded.
\end{rmk}
What about the converse? Is every weakly  bounded set  bounded?. Surprisingly, this is true and  is an important application of the Uniform Boundedness Principle.


\begin{prop}\label{wclosedsets}
Let $C$ be a convex set in a normed linear space $X$. Then $C$ is closed if and only if $C$ is weakly closed.
\end{prop}
\begin{proof}
Suppose that $C$ is weakly closed. Then $C^{c}$, the complement of $C$ is weakly open and hence open. Hence $C$ is closed.

On the other hand, let $C$ be closed. Let $x_0\notin C$. Then $d:=d(x_0,C)>0$. Let $D:={\{x\in X: d(x_0,x)<d}\}$. By the Hahn-Banach separation theorem, there exists $f\in X^*$ such that
\begin{equation*}
s:=\sup_{x\in C}f(x)<\inf_{x\in D}f(x).
\end{equation*}
Now, let us consider $U:={\{x\in X: f(x)>s}\}$. Then $U$ is weakly open  in $X$ and $U\cap C=\emptyset$. Thus $x_0\in U\subset C^{c}$. Therefore $C^{c}$ is weakly open, consequently, $C$ is closed.
\end{proof}

\begin{cor}
Let $X$ be a Banach space and $x_1,x_2,\dots,\in X$. If $x_n\xrightarrow{w}x$, then there exists a convex combination $y_k$ of $(x_n)$ such that $y_k\rightarrow x$ as $k\rightarrow \infty$.
\end{cor}
\begin{proof}
Let $C=\overline{co}{\{x_n:n\in \mathbb N}\}$, the closure of the convex hull of ${\{x_n:n\in \mathbb N}\}$. Then clearly, $C$ is a closed convex set. Now the conclusion follows by Proposition \ref{wclosedsets}.
  \end{proof}

\begin{thm}\label{equalityoftop}
Let $(X,\|\cdot\|)$ be a finite dimensional normed linear space. Then the weak and the norm topologies on $X$ are the same.

\end{thm}
\begin{proof}
Since every weak topology is weaker than the norm topology, every weak open set is open in the norm topology.

To prove the other way, let $U$ be an open set in the norm topology.

Let ${\{e_1,e_2,\dots,e_n}\}$ be a basis for $X$. If $x\in X$, there exists unique $(\alpha_1,\alpha_2,\dots, \alpha_n)\in \mathbb R^n$ such that
$x=\displaystyle \sum_{i=1}^n \alpha_ie_i$. Define $\|x\|_{\infty}=\displaystyle \max_{1\leq i\leq n}\|\alpha_i\|$. It is easy to check that $\|\cdot\|_{\infty}$ is a norm on $X$. Since all norms on a finite dimensional space are equivalent, the topology given by $\|\cdot\|_{\infty}$ and $\|\cdot\|$ are the same. Since $U$ is open in the topology given by $\|\cdot\|$, it is open in the topology given by $\|\cdot\|_{\infty}$. This means that given $x\in U$, there exists $\epsilon>0$ such that $B_{\infty}(x,\epsilon)\subset U$, where
\begin{align*}
B_{\infty}(x,\epsilon)&={\{y\in X: \|x-y\|_{\infty}<\epsilon}\}\\
                       &={\{y\in X: |y_i-x_i|<\epsilon,\; \text{for all}\; i=1,2,\dots,n}\}.
                       \end{align*}
Note that $U=\displaystyle \bigcup_{x\in U} B_{\infty}(x,\epsilon)$.

Now define $f_i:X\rightarrow \mathbb R \; (i=1,2,\dots,n)$ by
\begin{equation*}
  f_i(x)=\alpha_i,\; \text{for all}\;  x=\displaystyle \sum_{j=1}^n\alpha_je_j.
\end{equation*}
Then each $f_i\in X^*$ and

\begin{align*}
B_{\infty}(x,\epsilon)&={\{y=(y_i)\in X:|f_i(y-x)|<\epsilon,\; \text{for all}\; i=1,2,\dots,n}\}\\
                      &=B(x,f_1,f_2,\dots,f_n,\epsilon).\qedhere
\end{align*}
Therefore $U$ is weakly open.
\end{proof}

Next, we show that if $X$ is an infinite dimensional normed linear space, then the weak and the norm topologies are different. First we prove some basic results needed from Linear algebra.
\begin{prop}\label{kernelcomposition}
Let $V_1,V_2,V_3$ be normed linear spaces over $\mathbb K$ and $f:V_1\rightarrow V_3, \; g: V_1\rightarrow V_2$ be linear maps. Then there exists a linear map $h: V_2\rightarrow V_3$ such that
$$f=hog \Leftrightarrow g^{-1}(0)\subseteq f^{-1}(0).$$
\end{prop}
\begin{proof}
Assume that $g^{-1}(0)\subset f^{-1}(0)$. Define $h: g(V_1)\rightarrow V_3$ by
\begin{equation*}
h(g(x))=f(x),\; \text{for all}\; x\in V_1.
\end{equation*}
First we claim that $h$ is well defined. Let $x_1,x_2\in V_1$ be such that $g(x_1)=g(x_2)$. Then $x_1-x_2\in g^{-1}(0)\subset f^{-1}(0)$. That is  $f(x_1)=f(x_2)$. Hence $h$ is well defined. Now, extend $h$ to a linear map on $V_2$. Again denote the extension by $h$. Then clearly, $h(g)=f$.

The other implication is clear.
\end{proof}
\begin{lemma}\label{linearcombinationoffnls}
Let $f_1,f_2,\dots,f_n$ be linear functional on a normed linear space $X$ such that $\displaystyle \bigcap_{i=1}^nf_i^{-1}(0)\subseteq f^{-1}(0)$. Then there exists  scalars $\alpha_1,\alpha_2,\dots,\alpha_n$
such that $f=\displaystyle \sum_{i=1}^n \alpha_if_i$.
\end{lemma}
\begin{proof}
Take  $V_1=X, \; V_2=\mathbb R^n,\; V_3=\mathbb R,\; f=f$ and $g(x)=(f_1(x), f_2(x),\dots, f_n(x)$. Applying Proposition \ref{kernelcomposition}, we get a linear map $h:\mathbb R^n\rightarrow \mathbb R$ such that $f(x)=h(g(x))$ for all $x\in X$. The map $h$ can be written as
\begin{equation*}
h(y)=\displaystyle \sum_{i=1}^n\alpha_iy_i \; \text{for some}\; \alpha_1,\alpha_2,\dots,\alpha_n \; \text{and for every}\; y=(y_i)\in \mathbb R^n.
\end{equation*}
Therefore $f(x)=\displaystyle \sum_{i=1}^n \alpha_if_i(x)$ for every $x\in X$.
\end{proof}

\begin{thm}\label{infinitedimtopsnotsame}
Let $X$ be an infinite dimensional normed linear space and $U$ be a weak open set in $X$. Then $U$ is not bounded.
\end{thm}
\begin{proof}
Let $\mathcal O$ be a non empty weak open set in $X$. Assume that $0\in \mathcal O$.  Choose $\epsilon>0$ and $f_1,f_2,\dots,f_n\in X^*$ such that $${\{x\in X: |f_i(x)|<\epsilon,\; i=1,2,\dots,n}\}\subset \mathcal O.$$ Let $N:={\{x\in X: f_i(x)=0,\; i=1,2,\dots,n}\}=\displaystyle \bigcap _{i=1}^nf_i^{-1}(0)$.

We claim that $N\neq {\{0}\}$. If $N={\{0}\}$, then for every $f\in X^*$, we have $N\subset f^{-1}(0)$. Hence there exists scalars $\alpha_1,\alpha_2,\dots,\alpha_n$ such that
$f=\sum_{i=1}^n\alpha_if_i$ for all $f\in X^*$, from which we can conclude that $X^*$ is finite dimensional. This is a contradiction. Hence there exists $x\in X$ such that $x\in N$. Let $\lambda \in \mathbb F$. Then $\lambda x\in N$. Thus $N$ is not bounded and $N\subset \big\{x\in X: |f_i(x)|<\epsilon,\; i=1,2,\dots,n \big\}\subset \mathcal O$. This concludes that $\mathcal O $ is unbounded.
\end{proof}

\begin{rmk}\label{fdmlcharoftops}
 If $X$ is infinite dimensional normed linear space, then the open unit disc ${\{x\in X: \|x\|<1}\}$ is open in the norm topology. But it cannot be open in the weak topology by Theorem \ref{infinitedimtopsnotsame}. Hence  we have the following:\\
  A normed linear space is finite dimensional if and only if the $w$-topology and the norm topology are the same.
\end{rmk}

\begin{thm}[Schur's Lemma]\label{Schurslemma}
In $(\ell^1,\|\cdot\|_1)$, the weak convergence and the norm convergence are the same. In other words, if $(x_n)\subset X$ is a sequence, then
$x_n\xrightarrow{w}x$ if and only if $x_n\rightarrow x$.
\end{thm}

\begin{rmk}
  \begin{enumerate}
  \item By Remark \ref{fdmlcharoftops} and Schur's Lemma, we can observe that the convergent sequences are not enough to describe the $w$-topology. Instead, we need convergent nets for this purpose.
      \item By Theorem \ref{Schurslemma}, we can conclude that $B_{\ell^1}$ is not compact.
      \end{enumerate}
\end{rmk}

\newpage
\section{Weak-Star topology}

Since $X^*$ is a normed linear space we can define $\sigma(X^*,X^{**})$, the weak topology on $X^*$ generated by $X^{**}$.  But we are interested in the $\sigma(X^*,X)$, the weak topology generated by $X$. This can be done as follows;
Note that $X$ can be embedded in $X^{**}$ by the canonical map $J: X\rightarrow X^{**}$ given by
\begin{equation}
J(x)=j_{x},
\end{equation}
where $j_{x}:X^*\rightarrow \mathbb F$ is defined by
\begin{equation}
j_{x}(f)=f(x)\; \text{ for all}\;  f\in X^*.
\end{equation}
By definition we have $\|j_{x}\|=\sup{\{|f(x)|: f\in S_{X^*}}\}=\|x\|$, by a consequence of the  Hahn Banach theorem. Thus $\|Jx\|=\|x\|$ for all $x\in X$. Hence $J$ is an isometry. We can conclude that $X$ is isometrically isomorphic with $J(X)$. If $J(X)=X^{**}$, we call $X$ to be reflexive.

By the above embedding we can treat $X$ as a subspace of $X^{**}$.

\begin{defn}
Let $X$ be a normed linear space. The topology $\sigma(X^*,X)$ is called the weak star topology ($w^*$-topology) on $X^*$.
\end{defn}

Note that a set $G$ is an open set in the $w^*$-topology of $X^*$ if and only if for every $g\in G$ there exists an $\epsilon>0$ and $x_1,x_2,\dots,x_n\in X$ such that  $$\big\{f\in X^*: |(f-g)(x_i)|<\epsilon \big\}\subset G.$$

\begin{rmk}
Note that the $w^*$-topology on $X^*$ is weaker than the weak topology on $X^{**}$, since $X\subseteq X^{**}$. If $X$ is reflexive, then the $w^*$-topology on $X^*$ and the $w$-topology on $X^*$ are the same.
\end{rmk}

If a sequence $(f_n) \subset X^*$ converges to $f\in X^*$ in $\sigma(X^*,X)$, then we express this by writing $f_n\xrightarrow{w^*}f$.
\begin{prop}
 Let $X$ be a normed linear space and  $(f_n) \subset X^*$. Then $f_n\xrightarrow{w^*}f\in X^*$ if and only if $f_n(x)$ converges to $f(x)$ for all $x\in X$.
\end{prop}
\begin{proof}
   Suppose that $f_n(x)\rightarrow f(x)$ for all $x\in X$. Let $\mathcal O$ be a $w^*$-open set containing $f$. There exists an $\epsilon>0$ and $x_1,x_2,\dots,x_m\in X$ such that
\begin{equation*}
\big\{g\in X^*:|f(x_i)-g(x_i)|<\epsilon,\; i=1,2, \dots, m \big\} \subset \mathcal O.
\end{equation*}
We know that $f_{n}(x_i)\rightarrow f(x_i)$ for $i=1,2,\dots m$. Choose $n_i\in \mathbb N$ such that

\begin{equation*}
|f_n(x_i)-f(x_i)|<\epsilon \; \text{ for all}\; n\geq n_i.
\end{equation*}
Let $N=\max{\{n_i:i=1,2,\dots,m}\}$. Then for each $i=1,2,\dots, m$, we have
\begin{equation*}
|f_n(x_i)-f(x_i)|<\epsilon,\; \text{for all}\; n\geq N.
\end{equation*}
Thus $f_n\in \mathcal O$ for all $n\geq N$. So $f_n\xrightarrow{w^*}f$.

On the other hand, Let $\mathcal O={\{g\in X^*: |f(x)-g(x)|<\epsilon}\}$ be a $w^*$-open set containing $f$. Since $f_n\xrightarrow{w^*}f$, there exists $n_0\in \mathbb N$ such that $f_n\in \mathcal O$ for all $n\geq n_0$. That means, $|f_n(x)-f(x)|<\epsilon$ for all $n\geq n_0$. Hence $f_n(x)\rightarrow f(x)$. This is true for all $x\in X$.
\end{proof}

\begin{thm}\label{equalityoftop}
Let $X$ be a finite dimensional normed linear space. Then
Then the $w^*$-topology on $X^*$ and the norm topology on $X^*$ coincide
\end{thm}
\begin{proof}
We know that the $w^*$-open set is open in the norm topology. Let $\mathcal O$ be an open set and $f_0\in \mathcal O$. Choose $\epsilon>0$ such that $f_0+\epsilon B_{X^*}\subset \mathcal O$. Let ${\{e_1,e_2,\dots, e_n}\}$ be a basis for $X$. Define
\begin{equation*}
|||f|||:=\displaystyle \max_{1\leq i\leq n} |f(e_i)|,\; f\in X^*.
\end{equation*}
Then $|||\cdot |||$ is a  norm on $X^*$. Since all norms on a finite dimensional normed linear space are equivalent, there exists a $\delta>0$ such that $|||f|||\leq \delta$, we have $\|f\|<\epsilon$. Then the
$w^*$-open set $$ {\{f: \max_{1\leq i\leq n} |f(e_i)-f_0(e_i)|<\delta}\}$$ is contained in ${\{f: \|f-f_0\|<\epsilon}\}$. Hence $\mathcal O$ is $w^*$- open.
\end{proof}

Here we list out some more properties of the $w^*$-topology:

\begin{prop}
Let $X$ be a normed linear space and $(f_n)\subset X^*$. Then the following hold;

\begin{enumerate}
\item if $f_n\rightarrow f$ in $X^*$, then $f_n\xrightarrow{w^*}f$.
\item if $f_n\xrightarrow{w^*}f$, then $(f_n)$ is bounded and $\|f\|\leq \liminf_{n}\|f_n\|$.
\item if $f_n\xrightarrow{w^*}f$ and $(x_n)\subset X$ is such that $x_n\rightarrow x\in X$, then $f_n(x_n)\rightarrow f(x)$ as $n\rightarrow \infty$.
\end{enumerate}
\end{prop}
\begin{rmk}
  The $w^*$-topology on $X^*$ is Hausdorff.
\end{rmk}

\section{Banach-Alaouglu theorem}

In this section we prove the Banach-Alaouglu's theorem.
\begin{thm}[Alaouglu's theorem]
Let $X$ be a normed linear space and $B_{X^*}:={\{f\in X^*: \|f\|\leq 1}\}$. Then $B_{X^*}$ is $w^*$-compact in $X^*$.
\end{thm}
\begin{proof}
Let $x\in X$. Define $D_x={\{\lambda \in \mathbb K: |\lambda|\leq \|x\|}\}\subset \mathbb K$. Then $D_x$ is compact. By the Tychonoff's theorem the set $D=\Pi_{x\in X}D_x$ is compact in the product topology. Let $B^*$ denote the closed ball $B_{X^*}$ endowed with the $w^*$-topology. Define $\phi: B^*\rightarrow D$ by
\begin{equation*}
\phi(f)=(f(x))_{x\in X},\; \text{for all}\; f\in B^*.
\end{equation*}
We claim that $\phi$ is one-to-one and continuous. Clearly, $\phi$ is linear. If $(\phi(x))_{x\in X}=(0)$, then $f(x)=0$ for all $x\in X$. Hence $f=0$, that is, $\phi$ is one-to-one.

To show $\phi$ is continuous, let $(f_{\alpha})\subset B^*$ be such that $f_{\alpha}\xrightarrow{w^*}f$. Then $f_{\alpha}(x)\rightarrow f(x)$ for all $x\in X$. Consequently,
\begin{equation*}
\phi(f_{\alpha})=(f_{\alpha}(x))_{x\in X}\rightarrow (f(x))_{x\in X}=\phi(f).
\end{equation*}
This shows that $\phi$ is continuous.

If we show that $\phi(B^*)$ is  a closed subset of $D$ and $D$ being compact and Hausdorff, we can conclude that $\phi(B^*)$ is compact. So our next task is to show $\phi(B^*)$ is closed.
Let $\xi=(\xi_u)\in D$ be such that $\xi\in \overline{\phi(B^*)}$.  Define  $f: X\rightarrow \mathbb C$ by
\begin{equation*}
f(u)=\xi_{u},\; \text{for all}\; u\in X.
\end{equation*}
To show that $f$ is linear, let $x,y\in X$ and $\alpha, \beta \in \mathbb K$. Then for each $n\in \mathbb N$, choose $f_n\in B^*$ such that $|f(x)-f_n(x)|<\frac{1}{3n}$  for all $x\in X$. Thus,
\begin{equation}\label{linearityofcomponent}
|f_n(x)-f(x)|+|f_n(y)-f(y)|+|f_n(\alpha x+\beta y)-f(\alpha x+\beta y)|<\frac{1}{n}.
\end{equation}
Linearity of $f_n$ and Equation \ref{linearityofcomponent} imply that $$f(\alpha x+\beta y)=\alpha f(x)+\beta f(y).$$ That is $f$ is linear. Since $|\xi_x|\leq \|x\|$, we can conclude that $\|f\|\leq 1$ or $f\in B^*$. Now by the
definition of $f$, we can conclude that $\xi=\phi(f)\in \phi(B^*)$. Hence $\phi(B^*)$ is closed.

Since $B^*$ and $\phi(B^*)$ are homeomorphic, $B^*$ must be compact.
\end{proof}

\begin{thm}
Let $X$ be a normed linear space. Then $X$ is isometrically isomorphic to a subspace of $C(K)$, where $K=(B_{X^*}, \sigma(X^*,X))$.
\end{thm}
\begin{proof}
Let $x\in X$  and $j_x\in X^{**}$ be such that $j_{x}(f)=f(x)$ for all $f\in X^*$. Let $e_x=j_{x}|_{B_{X^*}}$. Then clearly, $\|e_{x}\|=\|j_{x}\|=\|x\|$, by a consequence of the Hahn-Banach extension theorem.
Now define a map $\eta: X\rightarrow C(K)$ by
\begin{equation*}
\eta(x)=e_{x},\; \text{for all}\; x\in X.
\end{equation*}
For $x,y\in X$, we have $e_{x+y}=e_x+e_y$ and $e_{\alpha x}=\alpha e_x$ for $\alpha \in \mathbb K$. Using these observations, we can conclude that $\eta$ is linear. Since $\|e_x\|=\|x\|$ for any $x\in X$, it follows that $\|\eta(x)\|=\|x\|$.
That is, $\eta$ is an isometry. This means that $X$ and $\eta(X)$ are isometrically isomorphic. Hence we can identify $X$ with $\eta(X)$, which is a subspace of $C(K)$.
\end{proof}

\section{Reflexivity}
Let $X$ be a normed linear space  and  $Y=X^*$. Then we can discuss $\sigma(Y^*,Y^{**})$, the $w$-topology on $Y^*=X^{**}$ and $\sigma(Y^*,Y)=\sigma(X^{**},X^*)$, the $w^*$-topology on $X^{**}$. As $X\subseteq X^{**}$, we can discuss the $w$-topology on $X$ induced by $X^{**}$ and the $w^*$-topology on $X$ induced by $X^*$. We have the following result.

\begin{thm}\label{weakstarclosureofball}
Let $X$ be a Banach space. Then $\overline{B_X}^{w^*}=B_{X^{**}}$ with respect to the $w^*$-closure of $X^{**}$ .
\end{thm}
\begin{proof}
First we show  that $B_X\subseteq B_{X^{**}}$.  Let $x\in B_{X}$. Then $j_x\in X^{**}$ and since $\|j_x\|=\|x\|$, it follows that $j_x\in B_{X^{**}}$. There fore $B_{X}\subseteq B_{X^{**}}$.
Since $B_{X^{**}}$ is $w^*$-compact, it follows that $\overline{B_{X^{**}}}^{w^*}=B_{X^{**}}$ with respect to the $w^*$ topology of $X^{**}$. Hence $\overline{B_{X}}^{w^*}\subset B_{X^{**}}$.

Assume that $B_{X}$ is not $w^*$-dense in $B_{X^{**}}$. Then there exists $\xi_0\in B_{X^{**}}$ and a $w^*$-neighbourhood $W$ of $\xi_0$ such that $W\cap B_{X}=\emptyset$. Note that $W={\{\xi\in X^{**}: |f_i(\xi)-f_i(\xi_0)|<\epsilon}\}$ for some $\epsilon>0$ and $f_i\in X^{*}$, for $i=1,2,\dots,n$. As $W\cap B_{X}=\emptyset$, there exists $i\in{\{1,2,\dots,n}\}$ such that
\begin{equation}\label{weakstardenseeqn}
  |f_i(x)-f_i(\xi_0)|\geq \epsilon,\quad (x\in X).
\end{equation}

Define $\Phi: X\rightarrow (\mathbb R^n,\|\cdot\|_{\infty})$ by
\begin{equation}
\Phi(x)=\Big(f_1(x),f_2(x),\dots,f_n(x) \Big),\; \text{for all}\; x\in X.
\end{equation}
Let $\alpha=(f_1(\xi_0),f_2(\xi_0),\dots, f_n(\xi_0))$. Then by Equation \ref{weakstardenseeqn}, we have that $\|\Phi(x)-\alpha\|_{\infty}\geq \epsilon$. That is $\alpha \notin \overline{\Phi(X)}$. Hence ${\{\alpha}\}$ and $\phi(X)$ can be separated by a hyperplane in $\mathbb R^n$. Hence by the Hahn-Banach Separation theorem, there exists $(\beta_1,\beta_2,\dots, \beta_n)\in \mathbb R^{n}$ ( As $(\mathbb R^{n})^*\simeq \mathbb R^{n}$ ) and $\beta \in \mathbb R^{n} $ such that
\begin{equation}\label{hyperplaneseparation}
\displaystyle \sum_{i=1}^n \beta_if_i(x)< \beta< \displaystyle \sum_{i=1}^n\beta_if_i(\xi_0)\leq \left\|\displaystyle \sum_{i=1}^n \beta_if_i\right\|.
\end{equation}

This implies that
\begin{equation*}
\left\|\displaystyle \sum_{i=1}^n \beta_if_i\right\|< \beta <  \left\|\displaystyle \sum_{i=1}^n \beta_if_i\right\|,
\end{equation*}
a contradiction. Hence our assumption that $W\cap B_{X}=\emptyset$ is wrong. This completes the proof.
\end{proof}
\begin{thm}
Let $X$ be normed linear space over $\mathbb R$.  Then $X$ is reflexive if and only if $B_{X}$ is $w$-compact.
\end{thm}

\begin{proof}
If $X$ is reflexive, then the weak topology and the weak star topology on $X^*$ same. By the Banach-Alaoglu theorem $B_{X^{**}}=B_{X}$ is $w^*$-compact. Hence $B_{X}$ is $w$-compact.

On the other hand, if $B_{X}$ is $w$-compact, then $B_{X}$ is $w^*$-compact and hence it must be $w^*$-closed. But by Theorem \ref{weakstarclosureofball}, $B_{X}$ is $w^*$ dense in $B_{X^{**}}$. Thus, $B_{X^{**}}=B_{X}$. Thus, we can conclude $X=X^{**}$.
\end{proof}
\begin{cor}\label{reflexivesubsp}
Let $Y$ be a closed subspace of a reflexive space $X$. Then $Y$ is reflexive.
\end{cor}
\begin{proof}
To show $Y$ is reflexive, it is enough to show that $B_Y$ is $w$-compact. Note that $B_Y=B_{X}\cap Y$ is closed and convex. So it is weakly closed subset of a weakly compact set $B_{X}$, hence it must be weakly compact with respect to $\sigma(X,X^*)$. But $\sigma(X,X^*)|_{Y}=\sigma(Y,Y^*)$, by the Hahn-Banach extension theorem. This concludes the proof.
  \end{proof}

  \begin{thm}
    Let $X$ be a Banach space. Then $X$ is reflexive if and only if $X^*$ is reflexive.
  \end{thm}
  \begin{proof}
    Suppose $X$ is reflexive. As $X=X^{**}$, the bounded linear functional on $X$ and the bounded linear functionals on $X^{**}$ are the same. That is $X^*=(X^{**})^{*}=(X^*)^{**}$. So $X^*$ is reflexive.
    On the other hand, if $X^*$ is reflexive, then by the above argument $X^{**}$ is reflexive. Since $X$ is a closed subspace of $X^{**}$, by Corollary \ref{reflexivesubsp}, $X$ is reflexive.
  \end{proof}

\begin{cor}
  Let $X$ be a reflexive Banach space and $C$ be a closed, convex, bounded set in $X$. Then $X$ is weakly compact.
\end{cor}
\begin{proof}
  Since $C$ is closed and convex, it must be weakly closed. Boundedness of $C$ implies that there exists $r>0$ such that $C\subset \bar{B}(0,r)={\{x\in X: \|x\|\leq r}\}$. As $\bar{B}(0,r)$ is weakly compact and $C$ being closed subset of $\bar{B}(0,r)$, it must be weakly compact.
\end{proof}
\begin{defn}
  Let $X$ be a normed linear space and $0\neq f\in X^*$. Then $f$ is said to be norm attaining if there exists $x_0\in S_{X}={\{x\in X: \|x\|=1}\}$
  such that $\|f\|=|f(x_0)|$.
\end{defn}

Note that if $X$ is a finite dimensional normed linear space , then every $f\in X^*$ attains its norm. This is because the unit ball $B_X$ is compact.

We have the following characterization of infinite dimensional reflexive spaces.

\begin{thm}[R. C. James]
  A Banach space $X$ is reflexive if and only if every $f\in X^*$ attains its norm.
\end{thm}
\begin{center}
Exercises
\end{center}
\begin{enumerate}
  \item Let $(x_n)$ be a sequence in a Banach space $X$. Show that $x_n\xrightarrow{w}x$ if and only if $(x_n)$ is bounded and the set ${\{f\in X^*: f(x_n)\rightarrow f(x)}\}$ is   dense in $X^*$.
  \item Let $X$ be a normed linear space and $(x_n)\subset X$. If $(x_n)$ is Cauchy and $x_n\xrightarrow{w}0$, then show that $x_n\rightarrow 0$.
  \item If $X$ is an infinite dimensional normed linear space, show that $0\in \overline{S_{X}}^{w}$.
  \item Let ${\{e_n:n\in \mathbb N}\}$ be an orthonormal basis for a Hilbert space $H$, then $e_n\xrightarrow{w}0$.
  \item Let $X=\ell^1\cong c_0^{*}$. Show that $e_n\xrightarrow{w^*}0$ but not $e_n\xrightarrow{w}0$.
  \item Show ${\{e_n:n\in \mathbb N}\}$, the standard orthonormal basis for $\ell^2$ is closed but not weakly closed. Show that ${\{e_n:n\in \mathbb N}\}\cup{\{0}\}$ is weakly compact in $\ell^2$.
  \item Let $f_n\in C[0,1]$ with $\|f_n\|_{\infty}\leq 1$ for each $n\in \mathbb N$. Show that $f_n\xrightarrow{w}o$ if and only if $f_n(x)\rightarrow 0$ for every $x\in [0,1]$.
  \item Let $X$ be a Banach space. Assume that $x_n\rightarrow x$ and $f_n\in X^*$ such that $f_n\xrightarrow{w^*}f$. Show that $f_n(x_n)\rightarrow f(x)$ as $n\rightarrow \infty$.

      \item Let $X$ be a Banach space and let $Y$ be a closed subspace of $X$. Show that the weak topology on $Y$ is the topology    induced on $Y$ by the weak topology on $X$.

\item Let $X$ be a Banach space and let $Y$ be a subspace of $X$. Show that the closure of $Y$ under the weak topology coincide with $\overline{Y}$, the closure of $Y$ in the norm topology.

\item Let $C[a,b]$ be the space of all continuous functions on $[a,b]$. For each $x\in [a,b]$ the map $E_x (f)=f(x)$ is a map from $C[a,b]$ to $\mathcal{C}$ called the evaluation map. prove that the weak topology generated by $\{E_x:x\in [a,b]\}$ is same as the topology  of pointwise convergence on $C[a,b]$.

\item Let $X$ be a normed linear space and $S=\{f\in X^* : || f ||=1 \}$. Let $(x_n)$ be a sequence in $X$. Show that $(x_n)$ converges weakly to $x\in X$ iff $f(x_n)$ converges $f(x)$ for every $f\in S$.

\item Let $(x_n)$ be a sequence in a normed linear space $X$ and let $x\in X$ show that if $Ref(x_n)$ converges to $Ref(x)$ for every $f\in X^*$ then $(x_n)$ converges weakly to $x$.

\item Prove that weak topology on $l_p$, $1<p<\infty$ cannot be obtained from any metric.

\item In $l_p$, $1\leq p \leq \infty$,  prove that weak convergence implies coordinate convergence. What about the converse?.

\item Let $(f_n)$ be a sequence in $X^*$. Show that $f_n\xrightarrow{w^*}f$ if and only if $(f_n)$ is bounded and the set ${\{x\in X: f_n(x)\rightarrow f(x)}\}$ is   dense in $X$.
\item Let $X$ be a normed linear space. Show that $\overline{S_X}^{w}=B_X$.
\end{enumerate}

\end{document}